\documentclass[a4paper,11pt]{article}
\usepackage[all]{xy}
\title{On o-minimal homotopy groups}

\author{El\'ias Baro \thanks{Partially supported by
GEOR MTM2005-02568}\\
Departamento de Matem\'aticas\\ Universidad Aut\'onoma de Madrid\\ 28049 Madrid, Spain \and Margarita Otero\thanks{Partially supported by
GEOR MTM2005-02568 and Grupos UCM 910444}\\ Departamento de Matem\'aticas\\ Universidad Aut\'onoma de Madrid\\28049 Madrid, Spain}

\usepackage{amsmath}
\usepackage{amscd}
\usepackage{amsfonts}
\usepackage{amssymb}
\usepackage{amsthm}

\newtheorem{deff}{Definition}[section]
\newtheorem{teo}[deff]{Theorem}
\newtheorem{lema}[deff]{Lemma}
\newtheorem{cor}[deff]{Corollary}
\newtheorem{prop}[deff]{Proposition}
\newtheorem{fact}[deff]{Fact}
\theoremstyle{definition}

\newtheorem{proper}[deff]{Properties}
\newtheorem{obs}[deff]{Remark}

\newcommand{\V}{\textrm{Vert}}
\newcommand{\Vc}{\emph{Vert}}

\newcommand{\St}{\textrm{St}}
\newcommand{\Stc}{\emph{St}}

\newcommand{\I}{\textrm{int}}

\newcommand{\co}{\textrm{co}}

\newcommand{\R}{\mbox{${\mathbb  R}$}}                  

\newcommand{\bt}{\begin{theorem}}
\newcommand{\et}{\end{theorem}}
\newcommand{\bl}{\begin{lemma}}
\newcommand{\el}{\end{lemma}}
\newcommand{\bexa}{\begin{example}}
\newcommand{\eexa}{\end{example}}
\newcommand{\bexe}{\begin{exercise}}
\newcommand{\eexe}{\end{exercise}}
\newcommand{\bprop}{\begin{proposition}}
\newcommand{\eprop}{\end{proposition}}
\newcommand{\bp}{\begin{proof}}
\newcommand{\ep}{\end{proof}}
\newcommand{\bc}{\begin{corollary}}
\newcommand{\ec}{\end{corollary}}
\newcommand{\bd}{\begin{definition}}
\newcommand{\ed}{\end{definition}}
\newcommand{\br}{\begin{remark}}
\newcommand{\er}{\end{remark}}

\begin{document}
\begin{center}ON O-MINIMAL HOMOTOPY GROUPS\end{center}
\begin{center}\begin{footnotesize}EL\'IAS BARO \footnote{Partially supported by
GEOR MTM2005-02568.}\begin{scriptsize} AND  \end{scriptsize}MARGARITA OTERO\footnote{Partially supported by
GEOR MTM2005-02568 and Grupos UCM 910444.

\textit{Date}: August 6, 2008

}
\end{footnotesize}\end{center}
\begin{quote}
\begin{footnotesize}\begin{scriptsize}ABSTRACT\end{scriptsize}. We work over an o-minimal expansion of a real closed field. The o-minimal homotopy
groups of a definable set are defined naturally using definable continuous maps. We prove that any two semialgebraic maps which are definably homotopic are also semialgebraically homotopic. This result together with the study of semialgebraic homotopy done by H. Delfs and M. Knebusch allows us to develop an o-minimal homotopy theory. In particular, we obtain o-minimal versions of the Hurewicz theorems and the Whitehead theorem.
\end{footnotesize} 
\end{quote}

\section{Introduction}\label{sintro}
Many  aspects of  o-minimal topology have been studied in the last years. However, there has been a lack of development of o-minimal homotopy \linebreak --only the first homotopy group was considered--. The first aim of this paper is to fill this gap. On the other hand, a much more complete development of semialgebraic homotopy theory was carried out by H. Delfs and  M. Knebusch in \cite{85DK}. 
 
Let $\mathcal{R}$ be  an o-minimal expansion of a real closed field.  We shall prove that  any two semialgebraic maps which are definably homotopic are also semialgebraically homotopic (see Theorem \ref{teo:princi} for the full statement). To do this, we will follow the scheme of the proofs of the results in \cite{85DK},  however the core of their proofs cannot be adapted to our context since they make use of both the polynomial description of semialgebraic sets and the Lebesgue number, which are not available in the o-minimal setting. Instead,  we use the results on normal triangulations in o-minimal structures obtained in \cite{07pB}.  
 
By applying both Theorem \ref{teo:princi} and semialgebraic homotopy, we are able to develop o-minimal homotopy. In section \ref{shomotgr}, the (higher) homotopy groups are defined and their usual propeties are  proved. We also prove the following transfer result (Corollary \ref{cor:homtopiadefigualtopo}):   if $X$ is a semialgebraic set defined without parameters  and $X(\R)$ is its realization over the reals, then for each $n\ge 1$, the $n$th  o-minimal homotopy group of $X$ is isomorphic to the (classical) $n$th homotopy group of $X(\R)$ --this was done for the case $n=1$ in \cite{02BeO}--. \linebreak In section \ref{shurewicz} we prove  the o-minimal versions of Hurewicz theorems and  Whitehead theorem, which were the motivation for this paper.
 
M. Shiota has announced some unpublished related results also linking the semialgebraic and the o-minimal topology (personal communication).

The results of this paper are part of the first author's Ph.D. dissertation.

\section{Preliminaries}\label{sprelim}For the rest of the paper we fix an o-minimal expansion $\mathcal{R}$ of a real closed field $R$. We always take 'definable' to mean 'definable in $\mathcal{R}$ with parameters'. We take the order topology on $R$ and the product topology on $R^n$ for $n>1$. All maps are assumed to be continuous. 

Given a definable set $S$ and some definable subsets $S_1,\ldots,S_l$ of $S$ we say that $(K,\phi)$ is a triangulation in $R^p$ of $S$ partitioning $S_1,\ldots,S_l$ if $K$ is a simplicial complex formed by a finite number of (open) simplices in $R^p$ and $\phi:|K|\rightarrow S$ is a definable homeomorphism  such that each $S_i$ is the union of the images by $\phi$ of some simplices of $K$. Recall that given a subset $A$ of $|K|$, the star of $A$ in $K$, denoted by $\St_{K}(A)$, is the union of all the simplices $\sigma\in K$ such that $\overline{\sigma}\cap A \neq \emptyset$. 
If a set $X$ is definable with parameters in some structure $\mathcal{M}$ we denote by $X(M)$ the realization of $X$ in $\mathcal{M}$.

\begin{lema}[\textbf{o-minimal homotopy extension lemma}]\label{prop:ext.homotopia}Let $X$, $Z$ and $A$ be definable sets with $A\subset X$ closed in $X$. Let
$f:X\rightarrow Z$ be a definable map and $H:A\times I \rightarrow Z$ a definable homotopy such that $H(x,0)=f(x)$, $x\in A$. Then there exists a definable homotopy $G:X\times I \rightarrow Z$ such that $G(x,0)=f(x)$, $x\in X$, and $G|_{A\times I}=H$.
\end{lema}
\begin{proof}Let $(K,\phi)$ be a triangulation of $X$ partitioning $A$ and let $K_{A}=\{\sigma \in K: \phi(\sigma)\subset A \}$.
Observe that $|K_{A}|$ is closed in $|K|$. By  Theorem 5.1 in \cite{84DK}, there exists a semialgebraic retract $r:|K| \times I
\rightarrow (|K_{A}|\times I) \cup (|K| \times \{0\})$. This retract naturally induces a definable retract $r':X \times I \rightarrow (A \times I) \cup (X \times
\{0\})$. Let $H':(A\times
I) \cup (X \times \{0\})\rightarrow Z$ be the following definable map 
\begin{displaymath}H'(x,t)=\left\{
\begin{array}{l}
H(x,t) \text{  \  \ for all  \ \ } (x,t)\in A\times I,\\
f(x) \text{  \  \ for all  \ \ } (x,0)\in X\times \{0\}.\\
\end{array}\right.
\end{displaymath}
Then $G=H'\circ r'$ is the required homotopy.
\end{proof}
Recall that given a definable map $f:|K|\rightarrow |L|$ between the realizations of two simplicial complexes $K$ and $L$, with $K$ closed, we say that a simplicial map $g:|K|\rightarrow |L|$ is \textit{a simplicial approximation to $f$} if $f(\Stc_{K}(w))\subset \Stc_{L}(g(w))$ for each $w\in \Vc(K)$. Note that if a simplicial complex $L$ is the first barycentric subdivision of another one then every simplex of $\overline{L}$ whose vertices lies in $L$ is a simplex of $L$. As in the classical case of closed simplices we obtain the following.
\begin{obs}\label{rmk:simplicialapprox}Let $K$ and $L$ be simplicial complexes, with $K$ closed, and let $f:|K|\rightarrow |L|$ be a definable map.\\
(i) If $g:|K| \rightarrow |L|$ is a simplicial approximation to $f$ then $f$ and $g$ are canonically definably homotopic via the map $(x,s)\mapsto (1-s)f(x)+sg(x)$ for all $(x,s)\in |K|\times I$.\\
(ii) If $f$ satisfies the \emph{star condition}, i.e, there is $\varphi:\Vc(K)\rightarrow \Vc(L)$ such that  $f(\Stc_{K}(v))\subset \Stc_{L}(\varphi(v))$ for every vertex $v\in \Vc(K)$ and moreover $L$ is the first barycentric subdivision of some simplicial complex, then $f$ has a simplicial approximation, namely, the simplicial map induced by $\varphi$.
\end{obs}

We shall use the following notion introduced in \cite{07pB} (see Definition 1.3 therein). Given a closed simplicial complex $K$ in $R^{m}$ and $S_{1},\ldots,S_{l}$ definable subsets of $|K|$, $(K',\phi')$ is a \textit{normal triangulation of $K$ partitioning $S_{1},\ldots,S_{l}$}, if it satisfies the following conditions:\\
(i) $(K',\phi')$ is a triangulation of $|K|$, $\phi':|K'|\rightarrow |K|$, partitioning $S_{1},\ldots,S_{l}$ and all $\sigma\in K$,  \\
(ii) $K'$ is a subdivision of $K$ (in particular $|K'|=|K|$), and\\
(iii) for every $\tau \in K'$ and $\sigma \in K$, if $\tau \subset \sigma$ then $\phi'(\tau)\subset \sigma$.
\vspace{0.2cm}

\begin{fact}\label{fact:normal}(i)(Normal Triangulation Theorem) Let $K$ be a closed simplicial complex and let $S_{1},\ldots,S_{l}$ be definable subsets of $|K|$. Then there exists a normal triangulation $(K',\phi')$ of $K$ partitioning $S_{1},\ldots,S_{l}$.\\
(ii) Let $(K',\phi')$ be a normal triangulation of a closed simplicial complex $K$ partitioning the definable subsets $S_{1},\ldots,S_{l}$ of $|K|$. Then $\textrm{id}_{|K|}$ and $\phi'$ are canonically definably homotopic via the map $(x,s)\mapsto (1-s)x+s\phi'(x)$, for all $(x,s)\in |K'|\times I$.
\end{fact}
For the proof of (i) and (ii) see Theorem 1.4 and Theorem 1.1 in \cite{07pB}, respectively.

Extending a given triangulation is a technical tool used in the construction of triangulations (see Lemma II.4.3 in \cite{85DK}). We next prove that the extension process can be done preserving normality. We will make use of this tool in the proof of Theorem \ref{teo:princi}.
\begin{lema}\label{prop:extnormal}Let $K$ be a closed simplicial complex and $K_Z$ a closed simplicial subcomplex of $K$. Let $(K_0,\phi_0)$ be a normal triangulation of $K_Z$. Then there exists a normal triangulation $(K',\phi')$ of $K$ such that $K_0\subset K'$ and $\phi'|_{|K_0|}=\phi_0$.
\end{lema}
\begin{proof}Note that $|K_{0}|=|K_{Z}|$, since $K_0$ is a subdivision of $K_Z$. For every $m\geq
0$ we denote by $SK^{m}$ the closed complex which is the union of $K_{Z}$ and all the simplices of $K$ of dimension $\leq m$. We will show that there exists a normal triangulation $(K^{m},\phi^{m})$ of $SK^{m}$ such that
$K_{0} \subset K^{m}$ and $\phi^{m}|_{|K_{0}|}=\phi_{0}$. Hence for $m=dim(K)$ we will obtain the required normal triangulation.

For $m=0$ let $K^{0}$ be the union of $K_{0}$ and all vertices of $K$. Let $\phi^{0}$ be equal to $\phi_{0}$ on
$|K_{0}|$ and the identity on the vertices of $K$ that does not lie in $|K_{0}|$. Clearly $(K^{0},\phi^{0})$ is a normal triangulation of $SK^0$, $K_{0}\subset K^{0}$ and $\phi^{0}|_{|K_{0}|}=\phi_{0}$.

Suppose we have constructed $(K^{m},\phi^{m})$. Let $\Sigma_{m+1}$ be the collection of simplices
in $K\setminus K_{0}$ of dimension $m+1$. Hence, for every
$\sigma\in \Sigma_{m+1}$, $\partial \sigma$ is contained in $SK^{m}$. On the other hand, $K^m$ is a subdivision of $SK^m$ and so, for each $\sigma\in \Sigma_{m+1}$, there exists a finite collection of 
indices $J_{\sigma}$ and simplices $\tau^{\sigma}_j$ of
$K^{m}$, $j \in J_{\sigma}$, such that $\partial \sigma=
\dot{\bigcup}_{j\in J_{\sigma}}\tau^{\sigma}_j$. For each $j\in J_{\sigma}$ denote by $[\tau^{\sigma}_j,\hat{\sigma}]$ the cone over $\tau^{\sigma}_j$ with
vertex the barycenter $\hat{\sigma}$ of $\sigma$, that is, $[\tau^{\sigma}_j,\hat{\sigma}]=\{(1-t)u+t\hat{\sigma}:u\in\tau_{j}^{\sigma},t\in [0,1] \}$. 
For each $\sigma\in \Sigma_{m+1}$ and $j\in J_{\sigma}$ we define
$$\begin{array}{crcl}
h^{\sigma}_j:& [\tau^{\sigma}_j,\hat{\sigma}]& \rightarrow & \overline{\sigma} \\
 & (1-t)u+t\hat{\sigma}& \rightarrow & (1-t)\phi^{m}(u)+t\hat{\sigma}.
\end{array}$$
Note that $h^{\sigma}_j$ is well-defined because given $u\in
\tau^{\sigma}_j$ there exists a proper face $\sigma_{0}\in K$ of $\sigma$ such that
$\tau^{\sigma}_j\subset \sigma_{0}$ and therefore, since $\sigma_0\in SK^m$ and $(K^m,\phi^m)$ is a normal triangulation, we have that $\phi^{m}(u)\in \phi^{m}(\tau^{\sigma}_j)\subset \sigma_{0} \subset
\partial \sigma$. Hence $h^{\sigma}_j((1-t)u+t\hat{\sigma})\in \overline{\sigma}$ for all $t\in [0,1]$ and $u\in \tau_j^{\sigma}$. Note that the map $h^{\sigma}_j$ is injective and it is indeed continuous. Let $K^{m+1}$ be the collection of simplices
in $K^{m}$ further with the collection of simplices $(\tau^{\sigma}_j,\hat{\sigma})=\{(1-t)u+t\hat{\sigma}:u\in\tau^{\sigma}_j,t\in (0,1)\}$
and all their faces for $\sigma\in \Sigma_{m+1}$ and $\tau_j^{\sigma}$ as described above. Finally, let $\phi^{m+1}$ be the extension  of $\phi^m$ to $K^{m+1}$ such that $\phi^{m+1}|_{[\tau^{\sigma}_j,\hat{\sigma}]}=h^{\sigma}_j$. We show that $\phi^{m+1}$ is well-defined. It is enough to prove that for a fixed $\sigma\in \Sigma_{m+1}$,  the sets $h^{\sigma}_j((\tau^{\sigma}_j,\hat{\sigma}))$ , $j\in J_{\sigma}$, are pairwise disjoint. Indeed, $h^{\sigma}_j((\tau^{\sigma}_j,\hat{\sigma}))=(\phi^{m}(\tau^{\sigma}_j),\hat{\sigma})$,
where
$(\phi^{m}(\tau^{\sigma}_j),\hat{\sigma})=\{(1-t)x+t\hat{\sigma}:x\in
\phi^{m}(\tau^{\sigma}_j), t\in (0,1)\}$ and  since the sets
$\phi^{m}(\tau^{\sigma}_j)$ are pairwise disjoint, the sets $(\phi^{m}(\tau^{\sigma}_j),\hat{\sigma})$ are also pairwise disjoint. Note that $\phi^{m+1}$ is continuous.

We now show that $(K^{m+1},\phi^{m+1})$ is a normal triangulation of $SK^{m+1}$. To prove that $(K^{m+1},\phi^{m+1})$ partitions the simplices of $SK^{m+1}$ it is enough to consider each $\sigma \in \Sigma_{m+1}$ (since $K^m\subset K^{m+1}$, $\phi^{m+1}|_{|K^m|}=\phi^m$, and $(K^m,\phi^m)$ is normal). Now, for each of these $\sigma\in \Sigma_{m+1}$, the image of $h_j^{\sigma}$ is contained in $\overline{\sigma}$ and, since $\partial \sigma=\bigcup_{j\in J_{\sigma}}\phi^{m}(\tau^{\sigma}_j)$, then we have that $\sigma=\bigcup_{j\in J_{\sigma}}(\phi^{m}(\tau^{\sigma}_j),\hat{\sigma})\cup  \{\hat{\sigma}\}$. Clearly $K^{m+1}$ is a subdivision of $SK^{m+1}$ because for the relevant simplices of $SK^{m+1}$, i.e, those $\sigma\in \Sigma_{m+1}$, the cones
$(\tau^{\sigma}_j, \hat{\sigma})$ and their faces form a triangulation
 of $\overline{\sigma}$. Also property $(iii)$ of normality holds, since we have always worked inside each simplex $\sigma \in
\Sigma_{m+1}$.
\end{proof}

\section{The o-minimal homotopy sets}\label{shomotset}
Let $(X,A)$ and $(Y,B)$ be two pairs of definable sets. Let $C$ be a relative closed definable subset of $X$ and let $h:C\rightarrow Y$ be a definable map such that $h(A\cap
C )\subset B$. We say that two definable maps $f,g:(X,A)\rightarrow
(Y,B)$ with $f|_{C}=g|_{C}=h$, are \textbf{definably homotopic relative to $h$}, denoted by $f\thicksim_h g$, if there exists a definable map $H:(X \times I,A \times I)\rightarrow (Y,B)$ such that $H(x,0)=f(x)$, $H(x,1)=g(x)$ for all $x\in X$ and $H(x,t)=h(x)$ for all $x\in C$ and $t\in I$. The \textbf{o-minimal homotopy set of $(X,A)$ and $(Y,B)$ relative to $h$} is the set
$$[(X,A),(Y,B)]_{h}^{\mathcal{R}}=\{f: f:(X,A)\rightarrow (Y,B) \textrm{ definable in } \mathcal{R}, f|_{C}=h \}/\thicksim_h.$$
If $C=\emptyset$ we omit all references to $h$. We shall denote by $\mathcal{R}_0$ the field structure of the real closed field $R$ of our o-minimal structure $\mathcal{R}$. Note that if we take $\mathcal{R}$ to be $\mathcal{R}_0$  above, then we obtain the definition of a semialgebraic homotopy set (see Section 2 of Chapter 3 in \cite{85DK}).

Our main result is the following theorem.
\begin{teo}\label{teo:princi}Let $(X,A)$ and $(Y,B)$ be two pairs of semialgebraic sets with $X$ closed and bounded. Let $C$ be a closed semialgebraic subset of $X$ and
$h:C\rightarrow Y$ a semialgebraic map such that $h(A \cap
C )\subset B$. Then, if $A$ is closed in $X$, the map
$$\begin{array}{rcl}
\rho:[(X,A),(Y,B)]_{h}^{\mathcal{R}_{0}} & \rightarrow &
[(X,A),(Y,B)]_{h}^{\mathcal{R}} \\

[f] & \mapsto & [f]
\end{array}$$
is a bijection.
\end{teo}

We are specially interested in the case $C=\emptyset$. However, in order to reduce Theorem \ref{teo:princi} to the following proposition, we will need to consider the general case.
\begin{prop}\label{prop:reduc4}Let $K$, $K_C$ and $L$ be simplicial complexes with $K$ closed. Let $h:|K_C| \rightarrow |L|$ be a semialgebraic map, where $K_C$ is a closed subcomplex of $K$. Then the map
$$\rho:[|K|,|L|]_{h}^{\mathcal{R}_{0}}\rightarrow [|K|,|L|]_{h}^{\mathcal{R}}$$
is surjective.
\end{prop}
Granted the proposition, Theorem 3.1 is proved as follows. We first make two reductions: (i) it suffices to prove that $\rho$ is onto and (ii) it suffices to consider the case $A=B=\emptyset$. These two reductions are done in the proof of Theorem 4.2 of Chapter III in \cite{85DK}, pp. 250. Though the statement of Theorem 4.2 in \cite{85DK} differs from our Theorem \ref{teo:princi} (because a real closed field $S$ extension of $R_0$ is considered there instead of our o-minimal $\mathcal{R}$), the proof of the mentioned reductions apply to our context except that at some point they use the semialgebraic homotopy extension lemma and we have to use our o-minimal homotopy extension lemma (see Lemma \ref{prop:ext.homotopia}). Finally, applying the semialgebraic Triangulation Theorem we can reduce to realizations of simplicial complexes.

\begin{proof}[Proof of Proposition \ref{prop:reduc4}]First note that we can assume that $L$ is the first barycentric subdivision of some simplicial complex. Let $[f]\in [|K|,|L|]_{h}^{\mathcal{R}}$. We will find a semialgebraic map definably homotopic to $f$ with relative to $h$. By the proof of Theorem 4.2 of Chapter III in \cite{85DK}, p.254, we can assume that 

(a) there exist two closed subcomplexes
$K_{D}$ and $K_{E}$ of $K$ such that $$|K_{C}|\subset \I_{|K|}(|K_{E}|)\subset|K_{E}|\subset \I_{|K|}(|K_{D}|),\textrm{ and}$$

(b) the map $f$ satisfies $f|_{|K_{D}|}=\widetilde{h}$, where
$\widetilde{h}:|K_{D}|\rightarrow |L|$ is a semialgebraic map such that $\widetilde{h}|_{|K_{C}|}=h$ and for each simplex $\sigma \in K_{D}$ there is a simplex of $L$ containing $\widetilde{h}(\sigma)$.

As above, even though the statement of Theorem 4.2 in \cite{85DK} differs from ours, since extensions of real closed fields are considered there, the proof of the fact that we can assume (a) and (b) apply to our context using the o-minimal homotopy extension lemma (see Lemma \ref{prop:ext.homotopia}) instead of the semialgebraic one. These assumptions allow us to protect $|K_C|$ with two successives ''barriers'', $|K_D|$ and $|K_E|$. We shall use these barriers in two different places in the following proof to transform the map $f$ without modifying it on $|K_C|$.

We divide the proof in two steps. In Step 1 we will make use of the Normal triangulation theorem (see Fact \ref{fact:normal}) to show that there exists a definable map $g$ satisfying the star condition such that $f\thicksim_h g$. In Step 2 we will use a simplicial approximation to $g$ (whose existence is ensured by the star condition) to find a semialgebraic map definably homotopic to $f$ relative to $h$.

\emph{Step 1: }Let $K_{Z}$ be the closed subcomplex of $K$ whose polyhedron is $|K_{Z}|=|K|\setminus \I_{|K|}(|K_{D}|)$. By the Normal Triangulation Theorem (see Fact \ref{fact:normal}) there exists a normal triangulation $(K_{0},\phi_{0})$ of $K_Z \cup K_E$ partitioning $f^{-1}(\sigma)\cap |K_{Z}|$, $\sigma\in L$. Moreover, since $|K_{E}|\cap |K_{Z}|= \emptyset$ and $(K_E,\text{id})$ is a normal triangulation of $K_E$, we can assume that $\phi_{0}|_{|K_{E}|}=\text{id}$. Next we extend $(K_0,\phi_0)$ to a triangulation of the whole of $|K|$. By Lemma \ref{prop:extnormal} there exists a normal triangulation $(K',\phi')$ of $K$ such that $K_{0}\subset K'$ and
$\phi'|_{K_{0}}=\phi_{0}$. In particular, $\phi'|_{|K_{E}|}=\text{id}.$
Note that $(K',\phi')$ partitions the sets $f^{-1}(\sigma)$,
$\sigma \in L$. Indeed, it suffices to show that for each $\sigma'\in K'$, $\phi'(\sigma')$ is contained in the preimage by $f$ of some simplex of $L$. If $\sigma'\subset |K_Z|$ this is clear since $\phi'$ extends $\phi_0$, which in turn partitions the subsets $f^{-1}(\sigma)\cap |K_{Z}|$ for $\sigma\in L$. On the other hand, if $\sigma'\subset |K|\setminus|K_{Z}|\subset |K_D|$ then $\phi'(\sigma')$ is contained in some simplex of $K_{D}$ because $(K',\phi')$ partitions the simplices of $K$ and, by (b), each simplex of $K_D$ is contained in the preimage by $f$ of some simplex of $L$.

By Fact \ref{fact:normal}, $\phi'$ and $id_{|K'|}$ are definably homotopic via the canonical homotopy $H_1:|K'|\times I\rightarrow |K|:(x,s)\mapsto (1-s)x+s\phi'(x)$. Let $H_2=f\circ H_1$. Since  $H_1(x,s)=x$ for all $x\in |K_{E}|$ and $t\in I$, $H_2$ is a homotopy between $f$ and $g:=f\circ \phi'$ relative to $\widetilde{h}|_{|K_{E}|}$. Note also that since $(K',\phi')$ partitions $f^{-1}(\tau)$ for $\tau\in L$ we have that for every $\sigma\in K'$ there exists $\tau\in L$ such that $g(\sigma)\subset \tau$. This implies that for every $v\in Vert(K')$ there exist $w\in Vert(L)$ with $w\in L$ such that $g(\Stc_{K'}(v))\subset \Stc_{L}(w)$. Indeed, take $v\in \V(K')$ and $\tau \in L$ such that $g(v)\in \tau$. Since $L$ is the first barycentric subdivision of some simplicial complex, there exists a vertex $w$ of $\tau$ with $w\in L$. Since $g^{-1}(\Stc_{L}(w))$ is the realization of a subcomplex of $K'$, it is open in $|K'|$ and contains the vertex $v$, we deduce that $\Stc_{K'}(v)\subset g^{-1}(\Stc_{L}(w))$.

\emph{Step 2: }Consider the map $\mu_{\textrm{vert}}:\V(K') \rightarrow \V(L):v\mapsto \mu_{\textrm{vert}}(v)$, where (as in Step 1) $\mu_{\textrm{vert}}(v)$ is such that $\mu_{\textrm{vert}}(v)\in L$ and $g(\St_{K'}(v))\subset \St_{L}(\mu_{\textrm{vert}}(v))$. By Remark \ref{rmk:simplicialapprox}(ii) the map $\mu_{\textrm{vert}}$ induces a simplicial approximation $\mu$ to $g$. However neither $\mu$ nor the canonical homotopy between $\mu$ and $g$ (see Remark \ref{rmk:simplicialapprox}) are good enough for us  since we need a map definably homotopic  to $f$ relative to $h$. We do as follows. Since $|K_{C}|$ and $|K|\setminus \I_{|K|}(|K_{E}|)$ are closed and disjoint, by Theorem 1.6 in \cite{84DK}, there exists a semialgebraic function $\lambda:|K| \rightarrow [0,1]$ such that
$\lambda^{-1}(0)=|K_{C}|$ and $\lambda^{-1}(1)=|K|\setminus \I_{|K|}(|K_{E}|)$.
Consider the map $H:|K|\times I \rightarrow  |L|:(x,s) \mapsto (1-s\lambda(x))g(x)+s\lambda(x)\mu(x)$.
The definable map $H$ is indeed continuous and is well-defined.
Note that
\begin{displaymath}\left\{
\begin{array}{l}
H(x,0)=g(x)\text{  \  \ for all  \ \ } x\in |K|,\\
H(x,s)=g(x)=h(x)\text{  \  \ for all  \ \ } x\in |K_{C}| \textrm{ and }s\in I.\\
\end{array}\right.
\end{displaymath}
Furthermore, observe that
\begin{displaymath}H(x,1)=\left\{
\begin{array}{l}
\mu(x)\text{  \  \ for all  \ \ } x\in |K|\setminus \I_{|K|}(|K_{E}|),\\
(1-\lambda(x))\widetilde{h}(x)+\lambda(x)\mu(x)\text{  \  \ for all  \ \ } x\in |K_{E}|,\\
\end{array}\right.
\end{displaymath}
is semialgebraic. Hence $f \thicksim_h g \thicksim_h H(x,1)$, as required.
\end{proof}

As an immediate consequence of Theorem \ref{teo:princi} we prove a more general result. 
\begin{cor}\label{cor:princi}Let $(X,A_{1},\ldots,A_{k})$ and $(Y,B_{1},\ldots,B_{k})$ be two systems of semialgebraic sets.  Let $C$ be a relative closed semialgebraic subset of $X$ and $h:C\rightarrow Y$ a semialgebraic map such that $h(C\cap A_{i})\subset B_{i}$, $i=1,\ldots,k$. Then, if the subsets $A_{1},\ldots,A_{k}$ are relative closed in $X$, the map
\begin{small}$$\begin{array}{rcl}
\rho:[(X,A_{1},\ldots,A_{k}),(Y,B_{1},\ldots,B_{k})]_{h}^{\mathcal{R}_{0}}&
 \rightarrow &
[(X,A_{1},\ldots,A_{k}),(Y,B_{1},\ldots,B_{k})]_{h}^{\mathcal{R}} \\

 [f] & \mapsto & [f]
\end{array}$$
\end{small}
is a bijection.
\end{cor}
\begin{proof}First note that the hypotheses of Lemma \ref{prop:ext.homotopia} do not include $X$ closed. Using that lemma (instead of its semialgebraic analogue) it can be shown (as in the proof of Theorem 3.1) that is enough to prove that $\rho_1:[X,Y]_{h}^{\mathcal{R}_{0}}\rightarrow [X,Y]_{h}^{\mathcal{R}}:[f]\mapsto [f]$
is surjective. Let $(K,\phi)$ be a semialgebraic triangulation of $X$ partitioning $C$. Now, let $K_{C}=\{\sigma \in K:\phi(\sigma)\subset C\}$. We may assume that $K$ is the first barycentric subdivision of some simplicial complex. Since $\phi$ is a semialgebraic homeomorphism, it is enough to prove that the map $\rho_2:[|K|,Y]_{\widetilde{h}}^{\mathcal{R}_{0}} \rightarrow
[|K|,Y]_{\widetilde{h}}^{\mathcal{R}}:[f]\mapsto [f]$ is surjective, where
$\widetilde{h}=h\circ \phi|_{|K_{C}|}$. Observe that since $K_{C}$
is a relative closed subcomplex of $K$ then $\co(K_{C})=\co(K)\cap K_C$ (recall that the core $\co(K)$ of a simplicial complex $K$ is the unique maximal subcomplex of $K$ whose realization is closed in the ambient space). Since $K$ is the first barycentric subdivision of some simplicial complex, $\co(K)$ and $\co(K_{C})$ are not empty. By Proposition III.1.6 in \cite{85DK}, there exists a semialgebraic strong deformation retract of $(|K|,|K_{C}|)$ to $\co(K,K_{C}):=(\co(K),\co(K_C))$. Now the proof of Theorem 4.2 of Chapter III in \cite{85DK}, pp.253, applies in our context (using Lemma \ref{prop:ext.homotopia} instead of the semialgebraic homotopy extension lemma) and therefore, it is enough to prove that $\rho_3:[|\co(K)|,Y]_{\hat{h}}^{\mathcal{R}_{0}} \rightarrow
[|\co(K)|,Y]_{\hat{h}}^{\mathcal{R}}:[f]\mapsto [f]$ is surjective, where $\hat{h}=\widetilde{h}|_{|\co(K_C)|}$. Finally, since $|\co(K)|$ is closed and bounded, $\rho_3$ is surjective by Theorem \ref{teo:princi}.
\end{proof}
\begin{cor}\label{cor:ominreal}Let $X$ and $Y$ be two pairs of semialgebraic sets defined without parameters. Then there exist a bijection
$$\rho:[X(\mathbb{R}),Y(\mathbb{R})] \rightarrow [X,Y]^{\mathcal{R}},$$
where $[X(\mathbb{R}),Y(\mathbb{R})]$ denotes the classical homotopy set. Moreover, if $R$ contains the real field, then the result remains true for all semialgebraic sets $X$ and $Y$  defined with parameters in $\mathbb{R}$.
\end{cor}
\begin{proof}By Theorem III.5.1 in \cite{85DK}, there exits a canonical bijection between $[X(\mathbb{R}),Y(\mathbb{R})]$ and the semialgebraic homotopy set over the real algebraic numbers $[X(\overline{\mathbb{Q}}),Y(\overline{\mathbb{Q}})]^{\overline{\mathbb{Q}}}$. By  Theorem III.4.1 in \cite{85DK}, there exists a canonical bijection  between $[X(\overline{\mathbb{Q}}),Y(\overline{\mathbb{Q}})]^{\overline{\mathbb{Q}}}$ and $[(X,A),(Y,B)]^{\mathcal{R}_0}$. The result then follows by Theorem \ref{teo:princi}. The proof of the second part is similar.
\end{proof}
\begin{obs}This corollary remains true for systems of semialgebraic sets satisfying the hypotheses of Corollary \ref{cor:princi}.
\end{obs}
\begin{cor}\label{cor:sinpar}Let $X$ and $Y$ be two definable sets defined without parameters. Then any definable map $f:X\rightarrow Y$ is definably homotopic to a definable map $g:X\rightarrow Y$ defined without parameters. If moreover $X$ and $Y$ are semialgebraic then g can also be taken semialgebraic.
\end{cor}
\begin{proof}By the Triangulation Theorem there are triangulations of $X$ and $Y$ defined without parameters and therefore it suffices to prove the case in which both $X$ and $Y$ are semialgebraic. By Theorem \ref{teo:princi}, $f$ is definably homotopic to a semialgebraic map $g_1$. Finally, by Theorem III.3.1 in \cite{85DK} applied to the real algebraic numbers $\overline{\mathbb{Q}}$ and $\mathcal{R}_0$, $g_1$ is semialgebraically homotopic to a semialgebraic map $g$ defined without parameters.
\end{proof}

\section{The o-minimal homotopy groups}\label{shomotgr}

We begin this section with a general discussion of homotopy groups in the o-minimal setting. Then we will relate the semialgebraic and the o-minimal homotopy groups via our Theorem \ref{teo:princi}. Finally, we will prove the usual properties related to homotopy in the o-minimal framework.

We will work with the category whose objects are the definable pointed sets, i.e., $(X,x_{0})$, where $X$ is a definable set with $x_{0}\in X$, and whose morphisms are the definable continuous maps between definable pointed sets. In a similar way, we define the categories of definable pointed pairs, i.e., $(X,A,x_0)$, where $X$ is a definable set, $A$ is a definable subset of $X$ and $x_0\in A$.

Let $(X,x_{0})$ be a definable pointed set. The \textbf{o-minimal homotopy group} of dimension $n$, $n\geq 1$, is the set
$\pi_{n}(X,x_{0})^{\mathcal{R}}=[(I^{n},\partial
I^{n}),(X,x_{0})]^{\mathcal{R}}$. We define
$\pi_{0}(X,x_{0})$ as the set of definably connected components of $X$. The \textbf{o-minimal relative homotopy group} of dimension $n$, $n\geq 1$, of a definable pointed pair $(X,A,x_{0})$ is the homotopy set
$\pi_{n}(X,A,x_{0})^{\mathcal{R}}=[(I^{n},I^{n-1},J^{n-1}),(X,A,x_{0})]^{\mathcal{R}}$,
where $I^{n-1}=\{(t_{1},\ldots,t_{n})\in I^{n}:t_{n}=0 \}$ and
$J^{n-1}=\overline{\partial I^{n}\setminus I^{n-1}}$. 

As in the classical case, we can define a group operation in the o-minimal homotopy groups $\pi_{n}(X,x_{0})^{\mathcal{R}}$ and 
$\pi_{m}(X,A,x_{0})^{\mathcal{R}}$ via the usual sum of maps for $n\geq 1$ and $m\geq 2$. Moreover, these groups are abelian for $n\geq 2$ and $m\geq 3$ (see pp. 340 and pp. 343 in \cite{02H}). Also, given a definable map between definable pointed sets (or pairs), we define the induced map in homotopy by the usual composing, which will be a homomorphism in the case we have a group structure.
It is easy to check that with these definitions of o-minimal homotopy group and induced map, both the absolute and relative o-minimal homotopy groups $\pi_n(-)$ are covariant functors (see pp. 342 in \cite{02H}).

As a consequence of Theorem \ref{teo:princi}, we deduce the following relation between the semialgebraic and the o-minimal homotopy groups.
\begin{teo}\label{rhoesismorfis}For every semialgebraic pointed set $(X,x_0)$ and every $n\geq 1$, the map $\rho:\pi_{n}(X,x_0)^{\mathcal{R}_{0}}\rightarrow \pi_{n}(X,x_0)^{\mathcal{R}}:[f]\mapsto [f]$, is a natural isomorphism. 
\end{teo}
\begin{proof}By Theorem \ref{teo:princi} $\rho$ is a bijection and its clearly a homomorphism. For the naturality condition, just observe that by definition the following diagram
\vspace{-0.3cm}
\begin{displaymath}
\xymatrix{ \pi_{n}(X,x_{0})^{\mathcal{R}_{0}} \ar[r]^{\psi_{*}}
\ar[d]_{\rho} &
 \pi_{n}(Y,y_{0})^{\mathcal{R}_{0}}     \ar[d]^{\rho}\\
\pi_{n}(X,x_{0})^{\mathcal{R}} \ar[r]_{\psi_{*}} &
\pi_{n}(Y,y_{0})^{\mathcal{R}}}
\end{displaymath}
commutes,  for every semialgebraic map $\psi:(X,x_{0})\rightarrow (Y,y_{0})$.
\end{proof}
\begin{obs}This last result remains true in the relative case and its proof is similar. 
\end{obs}

Moreover, the following result,  which follows from Corollary \ref{cor:ominreal}, is an extension of Theorem 1.1 in \cite{02BeO} (where the case $n=1$ is treated) and gives us a relation between the classical and the o-minimal homotopy groups.
\begin{cor}\label{cor:homtopiadefigualtopo}Let $(X,x_{0})$ be a semialgebraic pointed set defined without parameters. Then there exists a natural isomorphism between the classical homotopy group
$\pi_{n}(X(\mathbb{R}),x_{0})$ and the o-minimal homotopy group \linebreak
$\pi_{n}(X(R),x_{0})^{\mathcal{R}}$ for every $n\geq 1$.
\end{cor}
\begin{cor}\label{cinvextexp}The o-minimal homotopy groups are invariants under elementary extensions and o-minimal expansions.
\end{cor}
\begin{proof}It follows from the Triangulation Theorem and Theorem \ref{rhoesismorfis}. 
\end{proof}

\begin{obs}These last results remain true in the relative case and their proofs are similar. Moreover, the analogue of Corollary \ref{cinvextexp} is true for homotopy sets of definable systems satisfying the hypotheses of Corollary \ref{cor:princi}. 
\end{obs}

\begin{proper}\label{propert}The following properties of the o-minimal homotopy groups can be proved just adapting the proofs of their classical analogues and therefore we have not included them here. However, we give precise references of those proofs in the classical literature that can be adapted in an easy way (and, in particular, those that avoid the use of the path spaces with the compact-open topology).\\

\vspace*{-0.3cm}
\hspace{-0.6cm}1) \textit{The homotopy property:} It is immediate that given two definably homotopic maps $\psi,\phi:(X,A,x_{0})\rightarrow
(Y,B,y_{0})$, their induced homomorphisms $\psi_{*},\phi_{*}:\pi_{n}(X,A,x_{0})^{\mathcal{R}} \rightarrow
\pi_{n}(Y,B,y_{0})^{\mathcal{R}}$ are equal for every $n\geq 1$. Note that for $A=\{x_0\}$ and $B=\{y_0\}$ we have the absolute case.\\

\vspace{-0.3cm}
\hspace{-0.6cm}2) \textit{The exactness property:} Let $(X,A,x_0)$ be a pointed pair. For every $n\geq 2$ we define the \textbf{boundary operator} 
$$\begin{array}{rcl}
\partial: \pi_n(X,A,x_0)^{\mathcal{R}} & \rightarrow &
 \pi_{n-1}(A,x_0)^{\mathcal{R}}\\

[f] & \mapsto & [f|_{I^{n-1}}]
\end{array}$$
For $n=1$, we define $\partial([u])$, $[u]\in \pi_1(X,A,x_0)^{\mathcal{R}}$, as the definably connected component of $A$ which contains $u(0)$. It is easy to prove that the boundary operator is a natural well-defined homomorphism for $n>1$. Moreover, if we denote by $i:(A,x_0)\rightarrow (X,x_0)$ and $j:(X,x_0,x_0)\rightarrow (X,A,x_0)$ the inclusion maps, then the following sequence is exact
\begin{small}$$ \cdots \rightarrow \pi_n(A,x_0)\stackrel{i_*}{\rightarrow} \pi_n(X,x_0)\stackrel{j_*}{\rightarrow} \pi_n(X,A,x_0)\stackrel{\partial}{\rightarrow} \pi_{n-1}(A,x_0) \rightarrow \cdots \rightarrow\pi_0(A,x_0),$$
\end{small}
\hspace{-0.2cm}where the superscript $\mathcal{R}$ has been omitted. Indeed, by the triangulation theorem we can assume that $(X,A,x_0)$ is the realization of a simplicial complex with vertices in the real algebraic numbers. Then the exactness property follows from Corollary \ref{cor:homtopiadefigualtopo}, the obvious fact that $\partial$ commutes with the isomorphism defined there and the classical exactness property.\\

\vspace{-0.3cm}
\hspace{-0.6cm}3) \textit{The action of $\pi_1$ on $\pi_n$:} We can also define the usual action of $\pi_1(-)^{\mathcal{R}}$ on $\pi_n(-)^{\mathcal{R}}$. That is, given a pointed set $(X,x_0)$ and $[u]\in \pi_1(X,x_0)^{\mathcal{R}}$ there is a well-defined isomorphism $\beta_{[u]}:\pi_{n}(X,x_{0})^{\mathcal{R}} \rightarrow  \pi_{n}(X,x_{0})^{\mathcal{R}}$ which only depends on $[u]$. In a similar way, given a pointed pair $(X,A,x_0)$ and $[u]\in \pi_1(A,x_0)^{\mathcal{R}}$ there is a well-defined isomorphism $\beta_{[u]}:\pi_{n}(X,A,x_{0})^{\mathcal{R}} \rightarrow  \pi_{n}(X,A,x_{0})^{\mathcal{R}}$ which only depends on $[u]$. The existence of both actions can be proved just adapting what is done in pp. 341 and pp. 345 in \cite{02H} to the o-minimal setting. We briefly recall the construction of this action in the absolute case: given $[f]\in \pi_n(X,x_0)^{\mathcal{R}}$, we define $\beta_{[u]}([f]):=[H(t,1)]$, where $H:I^n \times I\rightarrow X$ is a definable homotopy such that $H(t,0)=f(t)$ for all $t\in I^{n}$ and $H(t,s)=u(s)$ for all $t\in \partial I^{n}$ and $s\in [0,1]$ (note that Lemma \ref{prop:ext.homotopia} ensures the existence of this homotopy).

We will need the following technical lemma in the proof of the o-minimal Hurewicz theorem (in Section \ref{shurewicz}). We have included here its easy proof because of its cumbersome notation.
\begin{lema}\label{lema:empujacurvas}Let $\psi:(X,x_{0}) \rightarrow (Y,y_{0})$ be a definable map between definable pointed sets and let $[u]\in \pi_{1}(X,x_{0})^{\mathcal{R}}$. Then for all $[f]\in \pi_{n}(X,x_0)^{\mathcal{R}}$, $\psi_{*}(\beta_{[u]}([f]))=\beta_{\psi_{*}([u])}(\psi_{*}([f]))$.
\end{lema}
\begin{proof}It is enough to observe that if $H:I^{n}\times I \rightarrow X$ is a definable homotopy such that
$H(t,0)=f(t)$ for all $t\in I^{n}$ and $H(t,s)=u(s)$ for all $t\in \partial I^{n}$ and $s\in I$, then $\psi\circ H:I^{n}\times I \rightarrow Y$ is a definable homotopy such that $\psi \circ H(t,0)=\psi \circ f(t)$ for all $t\in I^{n}$ and $\psi \circ H(t,s)=\psi \circ u(s)$ for all $t\in \partial I^{n}$ and $s\in I$. 
\end{proof}

\hspace{-0.6cm}4) \textit{The fibration property:} We say that a definable map $p:E\rightarrow B$ is a \textbf{definable Serre fibration} if it has the definable homotopy lifting property for every closed and bounded definable set $X$, i.e., if for each definable homotopy $H:X\times I\rightarrow B$ and each definable map $\widetilde{f}:X\rightarrow E$ with $p\circ \widetilde{f}(x)=H(x,0)$ for all $x\in X$, there exists a definable homotopy $\widetilde{H}:X\times I\rightarrow E$ with $p\circ \widetilde{H}=H$ and $\widetilde{H}(x,0)=\widetilde{f}(x)$ for all $x\in X$.
\begin{obs}\label{obs:Serre fibration}The definable homotopy lifting property for a closed simplex $\sigma$ is equivalent to the definable homotopy lifting property for $\sigma$ relative to $\partial \sigma$, i.e., if a definable homotopy $H:\sigma\times I\rightarrow B$ lifts to a definable homotopy $\widetilde{H}:\sigma\times I\rightarrow E$ starting with a given lift $\widetilde{H}_0$ ($\widetilde{f}$ above) of $H_0$, then $H$ also lifts to an $\widetilde{H}$ extending a given definable homotopy $\widetilde{H}:\partial \sigma \times I \rightarrow E$. Indeed, there is a semialgebraic homeomorphism of $\sigma\times I$ onto itself which carries $\sigma\times \{0\}$ homeomorphically onto $(\sigma \times \{0\})\cup (\partial \sigma \times I)$ (see the proof of Theorem III.3.1 in \cite{59Hu}). Therefore, the homotopy lifting property for closed simplices is equivalent to the homotopy lifting property for closed and bounded definable sets $X$ relative to closed subsets $A$ of $X$. For, by the triangulation theorem we can assume that $X$ is the realization of a (closed) simplicial complex and $A$ is the realization of a (closed) subcomplex of $X$. By induction over the skeleta of $X$ it suffices to construct a lifting over the closure of each open simplex contained in $X\setminus A$ at a time (and relative to the lifting constructed previously over its frontier).
\end{obs}
With the above remark it is easy to adapt  to the o-minimal setting the corresponding classical proof of the following fact (see Theorem 4.41 in \cite{02H}).
\begin{teo}[The fibration property]\label{teo:fibrationproperty} For every definable Serre fibration $p:E\rightarrow B$, the induced map $p_*:\pi_n(E,F,e_0)\rightarrow \pi_n(B,b_0)$ is a bijection for $n=1$ and an isomorphism for all $n\geq 2$, where $e_0\in F=p^{-1}(b_0)$. 
\end{teo}
As a consequence of the fibration property and the following proposition, we can extend Corollary 2.8 in \cite{04EO}, concerning coverings and the fundamental group, to all the homotopy groups (see Corollary \ref{cor:recubgrupos} below). For a definition of definable covering see Section 2 in \cite{04EO}.
\begin{prop}\label{prop:recufibracion}Every definable covering $p:E\rightarrow B$ is a definable Serre fibration. Moreover, every definable covering has the definable homotopy lifting property for definable sets.
\end{prop}
\begin{proof}Let $X$ be a definable set. Let $H:X\times I\rightarrow B$ a definable homotopy and $\widetilde{f}$ a definable map $\widetilde{f}:X\rightarrow E$ with $p\circ \widetilde{f}(x)=H(x,0)$ for all $x\in X$. Consider the definable family of paths $\{H_x:x\in X \}$, where $H_x:I\rightarrow B:t \mapsto H(x,t)$. Since $p$ has the path lifting property (see Proposition 2.6 in \cite{04EO}), for each $x\in X$ there is a (unique) lifting $\widetilde{H}_x:I\rightarrow E$ of $H_x$ such that $\widetilde{H}_x(0)=\widetilde{f}(x)$. Moreover, an easy modification of the proof of Proposition 2.6 in \cite{04EO} shows that the family of paths $\{\widetilde{H}_x:x\in X\}$ is definable. Therefore, the map $\widetilde{H}:X\times I\rightarrow E:(x,t)\mapsto \widetilde{H}_x(t)$ is definable, $p\circ \widetilde{H}=H$ and $\widetilde{H}(x,0)=\widetilde{f}(x)$ for all $x\in X$. It remains to prove that $\widetilde{H}$ is indeed continuous. Fix $(x_0,s_0)\in X\times I$. It is enough to prove that for each definable path $u:I\rightarrow X\times I$ with $u(1)=(x_0,s_0)$ we have that $\widetilde{H}(u(t))\rightarrow \widetilde{H}(x_0,s_0)$ when $t\rightarrow 1$. We will prove it for $s_0=1$, but the same proof works for every $s_0\in I$.\\

\vspace*{-0.3cm}
\hspace{-0.6cm}\textit{Claim:} We can assume that $u(0)=(x_0,0)$, that $u$ is definably homotopic to the canonical path $I\rightarrow X\times I:t\mapsto (x_0,t)$ and that $\widetilde{H}\circ u:[0,1)\rightarrow E$ is continuous.\\

\vspace*{-0.3cm}
\hspace{-0.6cm}Granted the Claim, the path homotopy lifting property of $p$ (see Proposition 2.7 in \cite{04EO}) implies that the respective liftings $\widetilde{H\circ u}$ and $\widetilde{H}_{x_0}$ of $H\circ u$ and $H_{x_0}$ starting at $\widetilde{f}(x_0)$, satisfy $\widetilde{H\circ u}(1)=\widetilde{H}_{x_0}(1)$. On the other hand, by the unicity of liftings of paths of $p$, we have that for every $\epsilon\in [0,1)$, $\widetilde{H}(u(t))=\widetilde{H\circ u}(t)$ for all $t\in [0,\epsilon]$. Therefore,  $\widetilde{H}(u(t))=\widetilde{H\circ u}(t)$ for all $t\in [0,1)$. Hence, $\widetilde{H}(u(t))\rightarrow \widetilde{H\circ u}(1)=\widetilde{H}_{x_0}(1)=\widetilde{H}(x_0,1)$ when $t\rightarrow 1$, as required.\\
\textit{Proof of the Claim:}  Since $X$ is definable, there exist a definably connected neighbourhood $U$ of $x_0$ which is definably contractible. Since $\widetilde{H}\circ u$ is definable, without loss of generality, $\widetilde{H}\circ u$ is continuous in $[\frac{2}{3},1)$ and $u(t)\in U\times I$ for all $t\in [\frac{2}{3},1)$. Denote by $u(\frac{2}{3})=(x_1,s_1)$ and take a definable path $w:[0,\frac{1}{3}]\rightarrow U$ such that $w(0)=x_0$ and $w(\frac{1}{3})=x_1$. We define the path $\hat{u}(t):I\rightarrow X\times I$ such that $\hat{u}(t):=(w(t),0)$ for all $t\in [0,\frac{1}{3}]$, $\hat{u}(t):=(x_1,3s_1(t-\frac{1}{3}))$ for all $t\in [\frac{1}{3},\frac{2}{3}]$ and $\hat{u}(t):=u(t)$ for all $t\in [\frac{2}{3},1]$. The definable path $\widetilde{H}(\hat{u}(t))$ is continuous for all $t\in [0,1)$ because $\widetilde{f}$ is continuous and because of the construction of $\widetilde{H}$. Since $U$ is definably contractible, $\{x_0\}\times I$ is a definable deformation retract of $U\times I$ and therefore $\hat{u}$ is definably homotopic to the canonical path $I\rightarrow X\times I:t\mapsto (x_0,t)$. Finally, since we are just interested in the behaviour of the definable path $u$ when $t$ is near $1$, we can replace $u$ by $\hat{u}$.
\end{proof}
\begin{cor}\label{cor:recubgrupos}Let $p:E\rightarrow B$ be a definable covering and let $p(e_0)=b_0$. Then $p_*:\pi_n(E,e_0)^{\mathcal{R}}\rightarrow \pi_n(B,b_0)^{\mathcal{R}}$ is an isomorphism for every $n>1$ and injective for $n=1$.
\end{cor}
\begin{proof}Since $p^{-1}(b_0)$ is finite, we have that $\pi_n(p^{-1}(b_0),e_0)=0$ for every $n\geq 1$. Then the result follows from Proposition \ref{prop:recufibracion} and both the exactness and the fibration properties.
\end{proof}
\end{proper}

\section{The o-minimal Hurewicz theorems and the o-minimal Whitehead theorem}\label{shurewicz}

Next we will prove both the absolute and relative Hurewicz theorems in the o-minimal setting by transferring from the semialgrebraic setting via Theorem \ref{teo:princi}.

First let us define the o-minimal Hurewicz homomorphism. Recall that, as it was proved in \cite{96W}, it is possible to develop an o-minimal singular homology theory $H_{*}(-)^{\mathcal{R}}$ on the category of definable sets. Moreover, by Proposition 3.2 in \cite{02BeO} there exists a natural isomorphism $\theta$  between the functors $H_{*}(-)^{\mathcal{R}_{0}}$ and
$H_{*}(-)^{\mathcal{R}}$ on the category of (pairs of) semialgebraic sets (note that the notation used in the above paper is different from ours, where $H_{*}(-)^{\mathcal{R}_{0}}=H_*^{sa}(-)$ and
$H_{*}(-)^{\mathcal{R}}=H_*^{def}(-)$ ). Fix $n\geq 1$. By Proposition 3.2 in \cite{02BeO}, $H_n(I^n,\partial I^n)^{\mathcal{R}_0}\cong H_n(I^n(\mathbb{R}),\partial I^n(\mathbb{R}))\cong \mathbb{Z}$. We fix a generator $z_n^{\mathcal{R}_0}$ of $H_n(I^n,\partial I^n)^{\mathcal{R}_0}$ and we define $z_n^{\mathcal{R}}:=\theta(z_n^{\mathcal{R}_0})$. Now, given a definable pointed set $(X,x_0)$, the \textbf{o-minimal Hurewicz homomorphism}, for $n\geq 1$, is the map $h_{n,\mathcal{R}}:\pi_{n}(X,x_{0})^{\mathcal{R}} \rightarrow H_{n}(X)^{\mathcal{R}}:[f] \mapsto h_{n,\mathcal{R}}( [f] )=f_{*}(z_{n}^{\mathcal{R}})$, where $f_{*}:H_{n}(I^{n},\partial I)^{\mathcal{R}}\rightarrow
H_{n}(X)^{\mathcal{R}}$ denotes the map in o-minimal singular homology
induced by $f$. We define the relative Hurewicz homomorphism adapting in the obvious way what was done in the absolute case. Now, following the classical proof, it is easy to check that $h_{n,\mathcal{R}}$ is a natural transformation between the functors $\pi_{n}(-)^{\mathcal{R}}$ and
$H_{n}(-)^{\mathcal{R}}$ (see Proposition V.4.1 in \cite{59Hu}). This fact can also be deduced from the semialgebraic setting (see Remark \ref{rhurewicz}).

The following result give us a relation between the semialgebraic and the o-minimal Hurewicz homomorphisms.
\begin{prop}\label{lema:homohurdefsa}Let $(X,x_{0})$ be a semialgebraic pointed set. Then the following diagram commutes
\begin{displaymath}
\xymatrix{ \pi_{n}(X,x_{0})^{\mathcal{R}_{0}}
\ar[r]^{h_{n,\mathcal{R}_{0}}} \ar[d]_{\rho} &
     H_{n}(X)^{\mathcal{R}_{0}} \ar[d]^{\theta}\\
\pi_{n}(X,x_{0})^{\mathcal{R}} \ar[r]_{h_{n,\mathcal{R}}} &
H_{n}(X)^{\mathcal{R}}}
\end{displaymath}
for all $n\geq 1$.
\end{prop}
\begin{proof}Let $[f]\in \pi_{n}(X,x_{0})^{\mathcal{R}_{0}}$. By definition $z_n^{\mathcal{R}}=\theta(z_n^{\mathcal{R}_0})$ and by the naturality of $\theta$ we have that $\theta(f_*(z_n^{\mathcal{R}_{0}}))=f_*(\theta(z_n^{\mathcal{R}_0}))$. Therefore $\theta(h_{n,\mathcal{R}_{0}}([f]))=\theta(f_*(z_n^{\mathcal{R}_{0}}))=f_*(\theta(z_n^{\mathcal{R}_0}))=f_*(z_n^{\mathcal{R}})=h_{n,\mathcal{R}}(\rho([f]))$.
\end{proof}
\begin{obs}\label{rhurewicz}(1) This last result remains true in the relative case and its proof is similar.\\
(2) Since $h_{n,\mathcal{R}_0}$ is a homomorphism for $n\geq 1$ (see Theorem III.7.3 in \cite{85DK}), it follows from Proposition \ref{lema:homohurdefsa} and the Triangulation theorem that $h_{n,\mathcal{R}}$ is also a homomorphism for $n\geq 1$.
\end{obs}
Recall the definition of the action of $\pi_1$ on $\pi_n$ defined in Properties \ref {propert}.3.
\begin{teo}[\textbf{The o-minimal Hurewicz theorems}]\label{teo:Hureabso}Let $(X,x_{0})$ be a definable pointed set and $n \geq 1$. Suppose that
$\pi_{r}(X,x_{0})^{\mathcal{R}}=0$ for  every $0\leq r \leq n-1$.
Then the o-minimal Hurewicz homomorphism $$h_{n,\mathcal{R}}:\pi_{n}(X,x_{0})^{\mathcal{R}}\rightarrow
H_{n}(X)^{\mathcal{R}}$$ is surjective and its kernel is the subgroup generated by $\{\beta_{[u]}([f])[f]^{-1}: [u]\in
\pi_{1}(X,x_{0})^{\mathcal{R}},[f]\in \pi_{n}(X,x_{0})^{\mathcal{R}}\}$. In particular, $h_{n,\mathcal{R}}$ is an isomorphism for $n\geq 2$.
\end{teo}
\begin{proof}Note that $X$ is definably connected since $\pi_0(X,x_0)^{\mathcal{R}}=0$. Let $(K,\phi)$ be a definable triangulation of
$X$ and $y_0=\phi^{-1}(x_0)$. Since $\pi_{r}(-)^{\mathcal{R}}$ is a covariant functor,
$\pi_{r}(|K|,y_{0})^{\mathcal{R}}=0$ for $0\leq r \leq n-1$. Moreover, as $\rho$ is a natural isomorphism, $\pi_{r}(|K|,y_{0})^{\mathcal{R}_{0}}\cong
\pi_{r}(|K|,y_{0})^{\mathcal{R}}=0$ for $0\leq r \leq n-1$. Since $h_{n,\mathcal{R}}$ is a natural transformation, the following diagram
\begin{displaymath}
\xymatrix{ \pi_{n}(|K|,y_{0})^{\mathcal{R}}
\ar[r]^{h_{n,\mathcal{R}}} \ar[d]_{\phi_{*}} &
     H_{n}(|K|)^{\mathcal{R}} \ar[d]^{\phi_*}\\
\pi_{n}(X,x_{0})^{\mathcal{R}} \ar[r]_{h_{n,\mathcal{R}}} &
H_{n}(X)^{\mathcal{R}}}
\end{displaymath}
commutes. Furthermore, since $\phi$ is a homeomorphism, the induced map $\phi_{*}$ in both homology and homotopy are isomorphism. Hence, by Lemma \ref{lema:empujacurvas}, it is enough to prove that
$h_{n,\mathcal{R}}:\pi_{n}(|K|,y_{0})^{\mathcal{R}}\rightarrow
H_{n}(|K|)^{\mathcal{R}}$ is surjective and that its kernel is the subgroup generated by $\{\beta_{[u]}([f])[f]^{-1}: [u]\in
\pi_{1}(|K|,y_{0})^{\mathcal{R}},[f]\in \pi_{n}(|K|,y_{0})^{\mathcal{R}}\}$. By Proposition \ref{lema:homohurdefsa}, the following diagram
\begin{displaymath}
\xymatrix{ \pi_{n}(|K|,y_{0})^{\mathcal{R}_{0}}
\ar[r]^{h_{n,\mathcal{R}_{0}}} \ar[d]_{\rho} &
     H_{n}(|K|)^{\mathcal{R}_{0}} \ar[d]^{\theta}\\
\pi_{n}(|K|,y_{0})^{\mathcal{R}} \ar[r]_{h_{n,\mathcal{R}}} &
H_{n}(|K|)^{\mathcal{R}}}
\end{displaymath}
commutes. Since $\rho$ and
$\theta$ are natural isomorphism, it is enough to prove that
$h_{n,\mathcal{R}_{0}}:\pi_{n}(|K|,y_{0})^{\mathcal{R}_{0}}\rightarrow
H_{n}(|K|)^{\mathcal{R}_{0}}$ is surjective and that its kernel is the subgroup generated by $\{\beta_{[u]}([f])[f]^{-1}: [u]\in
\pi_{1}(|K|,y_{0})^{\mathcal{R}_{0}},[f]\in
\pi_{n}(|K|,y_{0})\}^{\mathcal{R}_{0}}$. But this fact follows from the semialgebraic Hurewicz theorems (see Theorem III.7.4 in \cite{85DK}). Finally, the second part of the theorem follows immediately from the first one since for $n\geq 2$, by hypothesis, $\pi_1(X,x_{0})^{\mathcal{R}}=0$.
\end{proof}

\begin{teo}[The o-minimal relative Hurewicz theorems]\label{teo:Hurerel}Let $(X,A,x_{0})$ be a definable pointed pair and $n \geq 2$. Suppose that $\pi_{r}(X,A,x_{0})^{\mathcal{R}}=0$ for every $1\leq r \leq n-1$. Then the o-minimal Hurewicz homomorphism 
$h_{n,\mathcal{R}}:\pi_{n}(X,A,x_{0})^{\mathcal{R}}\rightarrow
H_{n}(X,A)^{\mathcal{R}}$ is surjective and its kernel is the subgroup generated by $\{\beta_{[u]}([f])[f]^{-1}: [u]\in
\pi_{1}(A,x_{0})^{\mathcal{R}},[f]\in \pi_{n}(X,A,x_{0})^{\mathcal{R}}\}$. In particular, $h_{n,\mathcal{R}}$ is an isomorphism for $n\geq 3$.
\end{teo}
\begin{proof}It is enough to adapt the proof of the o-minimal absolute Hurewicz theorems to the relative case. Note that at some point, we need the relative version of Lemma \ref{lema:empujacurvas} (whose proof is similar), i.e., that given a definable map $\psi:(X,A,x_{0}) \rightarrow (Y,B,y_{0})$ and $[u]\in \pi_{1}(A,x_{0})^{\mathcal{R}}$, we have that $\psi_{*}(\beta_{[u]}([f]))=\beta_{\psi_{*}([u])}(\psi_{*}([f]))$ for all $[f]\in \pi_{n}(X,A,x_0)^{\mathcal{R}}$.
\end{proof}
\begin{obs}1) With the hypotheses of Theorem \ref{teo:Hureabso}, for $n=1$, we have that $\text{Ker}(h_{1,\mathcal{R}})$ is the subgroup generated by
$\{\beta_{[u]}([v])[v]^{-1}: [u]\in
\pi_{1}(X,x_{0})^{\mathcal{R}},[v]\in \pi_{1}(X,x_{0})^{\mathcal{R}}\}$. On the other hand, using the definable homotopy $H(t,s)=u(ts)v(t)u(s-ts)$, we have that $\beta_{[u]}([v])[v]^{-1}=[u][v][u]^{-1}[v]^{-1}$. Hence, $\text{Ker}(h_{1,\mathcal{R}})$ is the commutator of $\pi_{1}(X,x_{0})^{\mathcal{R}}$. In particular, if \linebreak $\pi_{1}(X,x_{0})^{\mathcal{R}}$ is abelian then $h_{1,\mathcal{R}}$ is an isomorphism. This result was already proved in Theorem 5.1 in \cite{04EO}.\\
2)  With the hypotheses of Theorem \ref{teo:Hurerel}, $\pi_2(X,A,x_0)^\mathcal{R}/\text{Ker}(h_{2,\mathcal{R}}) \linebreak \cong H_2(X,A)^\mathcal{R}$ is abelian and therefore $\text{Ker}(h_{2,\mathcal{R}})$ contains the commutator subgroup of $\pi_2(X,A,x_0)^\mathcal{R}$. This fact can also be shown directly by proving that for every $[f],[g]\in \pi_2(X,A,x_0)^\mathcal{R}$, $[g][f][g]^{-1}=\beta_{[u]}([f])$,
where $u(t)=g(t,0)$ for $t\in I$.
\end{obs}

We finish this section with the proof of the o-minimal Whitehead theorem. We say that a definable map $\psi:X\rightarrow Y$ is a \textbf{definable homotopy equivalence} if there exist a definable map $\psi':Y\rightarrow X$ such that $\psi\circ \psi'\thicksim \textrm{id}_{Y}$ and $\psi'\circ \psi \thicksim \textrm{id}_{X}$. Note that if a definable map $\psi$ is a definable homotopy equivalence then it is a definable homotopy equivalence relative to a point. Indeed, it suffices to adapt the classical proof using Lemma \ref{prop:ext.homotopia} instead of the lifting property (see Proposition 0.19 in \cite{02H}).

\begin{teo}[\textbf{The o-minimal Whitehead theorem}]\label{whitehead}Let $X$ and $Y$ be two definably connected sets. Let $\psi:X\rightarrow Y$ be a definable map such that for some $x_0\in X$, $\psi_{*}:\pi_{n}(X,x_{0})^{\mathcal{R}}\rightarrow \pi_{n}(Y,\psi(x_{0}))^{\mathcal{R}}$ is an isomorphism for all $n\geq 1$. Then $\psi$ is a definable homotopy equivalence.
\end{teo}
\begin{proof}Let $(K,\phi_{1})$ and $(L,\phi_{2})$ be definable respective triangulations of $X$ and $Y$. Consider the points $x_1=\phi_1^{-1}(x_{0})$ and $y_1=\phi_2^{-1}(\psi(x_{0}))$. It suffices to prove that the definable map
$\widetilde{\psi}=\phi_{2}^{-1}\circ \psi \circ \phi_{1}:|K|\rightarrow |L|$
is a definable homotopy equivalence provided $\widetilde{\psi}_*:\pi_n(|K|,x_1)^{\mathcal{R}}\rightarrow \pi_n(|L|,y_1)^{\mathcal{R}}$ is an isomorphism for all $n\geq 1$.
By Theorem \ref{teo:princi} there exists a semialgebraic map $\varphi:(|K|,x_1)\rightarrow (|L|,y_1)$ such that $\varphi\thicksim \widetilde{\psi}$. 
By the homotopy property it follows that $\varphi_{*}=\widetilde{\psi}_{*}:\pi_{n}(|K|,x_1)^{\mathcal{R}}\rightarrow \pi_{n}(|L|,y_1)^{\mathcal{R}}$
is an isomorphism for all $n\geq 1$. Therefore by Theorem \ref{rhoesismorfis}, $\varphi_{*}:\pi_{n}(|K|,x_1)^{\mathcal{R}_{0}}\rightarrow \pi_{n}(|L|,y_1)^{\mathcal{R}_{0}}$
is an isomorphism for all $n\geq 1$. Hence, by Theorem III.6.6 in \cite{85DK}, $\varphi$ is a semialgebraic homotopy equivalence, that is, there exists a semialgebraic map $\varphi':|L|\rightarrow |K|$ such that $\textrm{id}_{|K|}\thicksim_0\varphi'\circ \varphi$ and $\textrm{id}_{|L|}\thicksim_0\varphi\circ \varphi'$, where $\thicksim_0$ means ``semialgebraically homotopic''. Hence $\textrm{id}_{|K|}\thicksim_0\varphi'\circ \varphi \thicksim \varphi' \circ \widetilde{\psi}$ and so $\textrm{id}_{|K|} \thicksim \varphi' \circ \widetilde{\psi}$.
In a similar way we prove that  $\textrm{id}_{|L|}\thicksim\widetilde{\psi}\circ \varphi'$. Therefore $\widetilde{\psi}$ is a definable homotopy equivalence, as required.
\end{proof}
\begin{cor}\label{cor:contrac}Let $X$ be a definable set and let $x_0\in X$. If $\pi_n(X,x_0)^{\mathcal{R}}=0$ for all $n\geq 0$ then $X$ is definably contractible.
\end{cor}
\begin{proof}It follows from Theorem \ref{whitehead} applied to a constant map.
\end{proof}
Next result follows the transfer approach developed in \cite{03BeO}.
\begin{cor}Let $X$ be a semialgebraic set defined without parameters. Then $X$ is definably contractible if and only if $X(\mathbb{R})$ is contractible in the classical sense. 
\end{cor}
\begin{proof}It follows from Corollary \ref{cor:homtopiadefigualtopo} and Corollary \ref{cor:contrac}.
\end{proof}

\vspace{0.6cm}
\hspace{-0.6cm}\begin{scriptsize}DEPARTAMENTO DE MATEM\'ATICAS, UNIVERSIDAD AUT\'ONOMA DE MADRID, 28049, MADRID, SPAIN.\end{scriptsize}\\
\begin{scriptsize}\textit{E-mail address}: \texttt{elias.baro@uam.es}\end{scriptsize}\\

\hspace{-0.6cm}\begin{scriptsize}DEPARTAMENTO DE MATEM\'ATICAS, UNIVERSIDAD AUT\'ONOMA DE MADRID, 28049, MADRID, SPAIN.\end{scriptsize}\\
\begin{scriptsize}\textit{E-mail address}: \texttt{margarita.otero@uam.es}\end{scriptsize}\\
\end{document}